\documentclass{amsart}

\usepackage{amsmath, amsthm, amssymb,slashed}

\usepackage[usenames, dvipsnames]{color}
\usepackage[svgnames]{xcolor}
\usepackage[colorlinks,citecolor=RoyalBlue, urlcolor=RoyalBlue, linkcolor=RoyalBlue ]{hyperref} 

\usepackage[normalem]{ulem}

\usepackage{booktabs}

\definecolor{mygray}{gray}{0.6}

\usepackage{upgreek}
\usepackage{bbm}

\newenvironment{myfont}[2][]{\csname#2\endcsname[#1]}{}

\usepackage{slashed}
\usepackage[makeroom]{cancel}
\usepackage[normalem]{ulem}
\usepackage{soul}
\newcommand{\stkout}[1]{\ifmmode\text{\sout{\ensuremath{#1}}}\else\sout{#1}\fi}

\usepackage{sseq}
\usepackage[all,cmtip]{xy}
\usepackage{tikz-cd}
\usepackage{tikz}
\usetikzlibrary{calc}
\usetikzlibrary{matrix}
\usetikzlibrary{decorations.markings}
\usetikzlibrary{tikzmark,decorations.pathreplacing,positioning}
\usepackage{amsfonts}

\newcommand{\bea}{\begin{eqnarray}}
\newcommand{\eea}{\end{eqnarray}}
\def\be{\begin{equation}}
\def\ee{\end{equation}}

\definecolor{red}{rgb}{1,0,0}
\definecolor{blue}{rgb}{0,0,1}
\definecolor{dblue}{rgb}{0,0,0.4}
\definecolor{green}{rgb}{0,1,0}
\definecolor{black}{rgb}{0,0,0}
\definecolor{white}{rgb}{1,1,1}

\definecolor{brn}{rgb}{.8,.4,.0}
\definecolor{redo}{rgb}{1,.5,.0}
\definecolor{ddgrn}{rgb}{0,0.4,0}
\definecolor{dgrn}{rgb}{0,0.55,0}
\definecolor{dbl}{rgb}{0,0,0.5}

\usepackage[bbgreekl]{mathbbol}
\usepackage{amscd}

\newcommand{\R}{\mathbb{R}}

\newcommand{\bpm}{\begin{pmatrix}}
\newcommand{\epm}{\end{pmatrix}}
\newcommand{\bmm}{\begin{matrix}}
\newcommand{\emm}{\end{matrix}}

\def\R{{\mathbb{R}}}

\newcommand {\emptycomment}[1]{}

\usepackage{centernot}
\newcommand{\nn}{{\nonumber}}
\newcommand{\Sec}[1]{Sec.~\ref{#1}}

\usepackage{enumitem} 
\usepackage{mathtools,amssymb,varwidth}

\usepackage{datetime}

\newtheorem{theorem}{Theorem}[section]

\newtheorem{lemma}[theorem]{Lemma}

\begin{document}

\title{The Dijkgraaf-Witten invariant is a partial case of the photography method}

\author{Vassily Olegovich Manturov}

\address{Moscow Institute of Physics and Technology, Moscow 141700, Russia \\
Nosov Magnitogorsk State Technical University, Zhilyaev Laboratory of mechanics of gradient nanomaterials, 38 Lenin prospect, Magnitogorsk, 455000, Russian Federation \\
vomanturov@yandex.ru}

\maketitle

\begin{abstract}
In \cite{ManturovNikonovMay2023,ManturovWanMay2023} the author discovered a very general principle (called {\em the photography principle}) which allows one:

a) To solve various equations 
(like pentagon equation)

b) To construct invariants of manifolds.

The advantage of that principle is that it deals with a very general notion of {\em data} and {\em data transmission} which may be of any kind.

In the present paper, we show that the definition of the Dijkgraaf-Witten invariants of manifolds can be thought of as an evidence of the above principle.

\end{abstract}



\section{Introduction}
A very general method, {\em the photography method} for solving the pentagon equation, and other equations, was proposed by the author in \cite{ManturovNikonovMay2023}. Roughly speaking, if one has several {\em states} (say, triangulations of a pentagon), some {\em data} for each state (e.g., edge lengths or triangle areas) and a {\em rule (formula)} for translating data from one state to another (say, Ptolemy equation), then after returning to the initial state, we obtain the same data.

Saying more standardly, this means that the photography method gives a solution to some equations in terms of ``data'' whatever this data means (in geometrical setting one can have lengths, angles, areas, volumes, etc.); even the amount of data may change, so, it is sometimes hard to say what we mean by an ``equation''.

We state the photography method as follows: ``We associate some data to each of the triangulations of an $n$-manifold. Then we use some data transmission law: 
if we know all for one triangulation, then we should know all for an adjacent triangulation, hence for any other triangulation.'' We may use this method to solve equations or construct invariants.

For $n\ge3$, $n$-manifolds can be described as equivalence classes of triangulations modulo Pachner moves \cite{Pachner1,Pachner2}. With a triangulation, one can associate some {\em data} and some {\em data transmission laws}: some elements in a group $G$ corresponding to 1-simplices and the values of a group $n$-cocycle $\phi\in H^n(G,A)$ ($A$ is an abelian group) corresponding to $n$-simplices, and the way these values are changing under moves.

Taking this together, we can get to an equation (a system of equations) needed to get a state-sum associated to triangulations, which will be invariant under, say, (2-3)-Pachner moves and (1-4)-Pachner moves. This equation is exactly the $n$-cocycle condition.

There are two versions of the photography method.
One of them is for solving equations \cite{ManturovWanMay2023}
and 
the other one is for constructing invariants \cite{KauffmanKimManturovNikonov}. 
In \cite{ManturovWanMay2023} we solve the pentagon equation in two ways: one by using the realisability of edge-lengths as lambda-lengths in the hyperbolic plane and Ptolemy relation, the other one by using the areas.
In \cite{KauffmanKimManturovNikonov} we construct invariants of manifolds by using the following observation.
If we have $(n+1)$ points (say, 5)
and all edges between them,
and for two $n$-tuples of points 
some relations hold,
then we can restore the last edge
so that the relations hold for all possible $n$-tuples.
For such relations, one can take
{\em realisability of edge-lengths as edge-lengths of a simplex in some Euclidean space}.

The invariants we have constructed before \cite{KauffmanKimManturovNikonov}
and here
are not exactly
``Turaev-Viro invariants'' \cite{TV} or
Dijkgraaf-Witten invariants \cite{DW}.
 
For example, in the paper \cite{ManturovNikonovMay2023}
instead of taking a finite palette of colours and solving the Biedenharn-Elliot
equations, we guess how to get a solution.
The main guess is the ``geometrical'' observation that
``if four quadrilaterals ABCD and ABCE are inscribed then
any of ABDE, ACDE, BCDE is inscribed''.
 
Similarly, in \cite[Sec. 3]{ManturovWanMay2023} 
we used ``areas'' of triangles in order to get some matrices satisfying pentagon identities.
But these ``areas'' are not areas in some proper sense: they are just variables which
satisfy the only condition:
the sum of ``areas'' of a whole polyhedron equals the sum of its constituent parts.
Then we take it more abstractly, and get a solution to the pentagon identity suggested
by Korepanov in \cite{K19}.

In the present paper a similar guess answers the question ``How to get Dijkgraaf-Witten-like invariants''?
We take a triangulation and associate data to its 1-dimensional edges and top-dimensional cells
in order to get an invariant under Pachner moves. The data we want to associate to edges are
the edge lengths, and the data associate to top-dimensional simplices are ``volumes''. When we try
to figure out what we really need from volumes, we see that this should be nothing but the cocycle
condition \eqref{eq:cocycle}.
A crucial observation explaining the ``naturality'' of the Dijkgraaf-Witten invariants is Lemma \ref{volume} which allows to axiomatize volumes as cocycles.

The paper is organized as follows. In \Sec{sec:definition}, we review the definition of the Dijkgraaf-Witten invariants. In \Sec{sec:main}, we show that the Dijkgraaf-Witten invariants are a partial case of the photography method. In \Sec{sec:further}, we list some directions of further research.

\subsection{Acknowledgements}

The author is grateful to L.A.Grunwald, L.H.Kauffman, Seongjeong Kim, V.G.Turaev, and Zheyan Wan for helpful discussion and comments.
The study was supported by the grant of Russian Science Foundation (No. 22-19-20073 dated March 25, 2022 ``Comprehensive study of the possibility of using self-locking structures to increase the rigidity of materials and structures''). 

\section{Dijkgraaf-Witten invariants}\label{sec:definition}
In the present section, we review the Dijkgraaf-Witten invariants for $n$-manifolds. For $n=3$, in \cite{DW} Dijkgraaf and Witten gave a combinatorial definition for Chern-Simons with finite gauge groups using 3-cocycles of the group cohomology. This was generalised to higher dimensional manifolds in \cite{Freed}. We follow the description in \cite{CKS}, see also \cite{Wakui} (we use the Pachner moves). 

Let $T$ be a triangulation of an oriented closed $n$-manifold $M$, with $a$ vertices and $m$ $n$-simplices. Give an ordering to the set of vertices. Let $G$ be a finite group. Let $\phi: \{\text{oriented edges}\}\to G$ be a map such that
\begin{enumerate}
\item 
for any triangle with vertices $v_0,v_1,v_2$ of $T$, $\phi(\langle v_0,v_2\rangle)=\phi(\langle v_0,v_1\rangle)\phi(\langle v_1,v_2\rangle)$, where $\langle v_i,v_j\rangle$ denotes the oriented edge with endpoints $v_i$ and $v_j$, and
\item
$\phi(-e)=\phi(e)^{-1}$.
\end{enumerate}

Let $\alpha: G^{\times n}\to A$, $(g_1,g_2,\dots,g_n)\mapsto \alpha[g_1|g_2|\cdots|g_n]\in A$, be an $n$-cocycle valued in a multiplicative abelian group $A$. The $n$-cocycle condition is 
\bea\label{eq:cocycle}
&&\alpha[g_2|g_3|\cdots|g_{n+1}]\alpha[g_1g_2|g_3|\cdots|g_{n+1}]^{-1}\alpha[g_1|g_2g_3|\cdots|g_{n+1}]\cdots\nn\\
&&\cdot\alpha[g_1|g_2|\cdots|g_ng_{n+1}]^{(-1)^n}\alpha[g_1|g_2|\cdots|g_n]^{(-1)^{n+1}}=1.
\eea
Then the Dijkgraaf-Witten invariant is defined by
\bea\label{eq:DW}
Z_M=\frac{1}{|G|^a}\sum_{\phi}\prod_{i=1}^mW(\sigma_i,\phi)^{\epsilon_i}.
\eea
Here $a$ denotes the number of the vertices of the given triangulation,
$W(\sigma,\phi)=\alpha[g_1|g_2|\cdots|g_n]$ where $\phi(\langle v_0,v_1\rangle)=g_1$, $\phi(\langle v_1,v_2\rangle)=g_2$, $\dots$, $\phi(\langle v_{n-1},v_n\rangle)=g_n$, for the $n$-simplex $\sigma=|v_0v_1v_2\cdots v_n|$ with the ordering $v_0<v_1<v_2<\cdots<v_n$, and $\epsilon=\pm1$ according to whether or not the orientation of $\sigma$ with respect to the vertex ordering matches the orientation of $M$.


\section{Dijkgraaf-Witten invariants are a partial case of the photography method}\label{sec:main}

We associate some data to each of the triangulations of an $n$-manifold. We may associate data to edges and $n$-simplices. Then we use some data transmission law: if we know all for one triangulation, then we should know all for an adjacent triangulation, hence for any other triangulation. 
The data are the elements of $G$ associated to the edges and the values of an $n$-cocycle on the $n$-simplices and we consider the state sum defined by \eqref{eq:DW}. 
The elements of $G$ associated to the edges of a triangle satisfy the product rule.
The data transmission law is the $n$-cocycle condition. Two adjacent triangulations are connected by a Pachner move. 

\begin{figure}[h]
\centering\includegraphics[width = 0.9\textwidth]{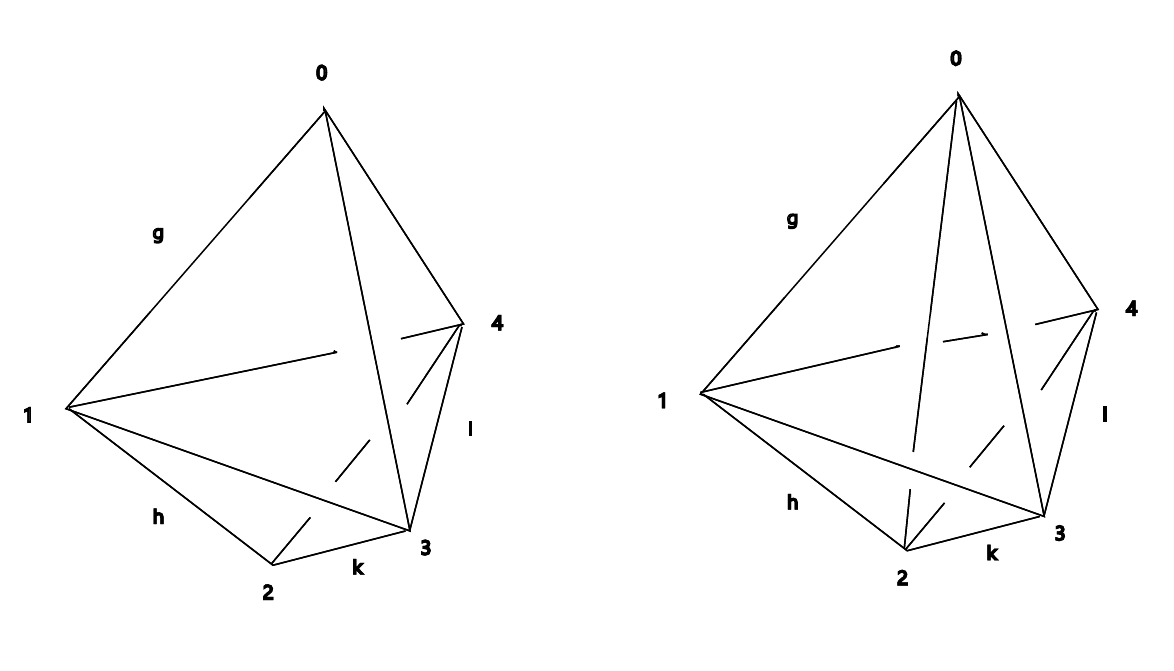}
\caption{Pachner 2-3 moves. The tetrahedra 0134 and 1234 on the left becomes the tetrahedra 0123, 1234, and 0124 on the right.}\label{fig:Pachner}
\end{figure}

\begin{figure}[h]
\centering\includegraphics[width = 0.9\textwidth]{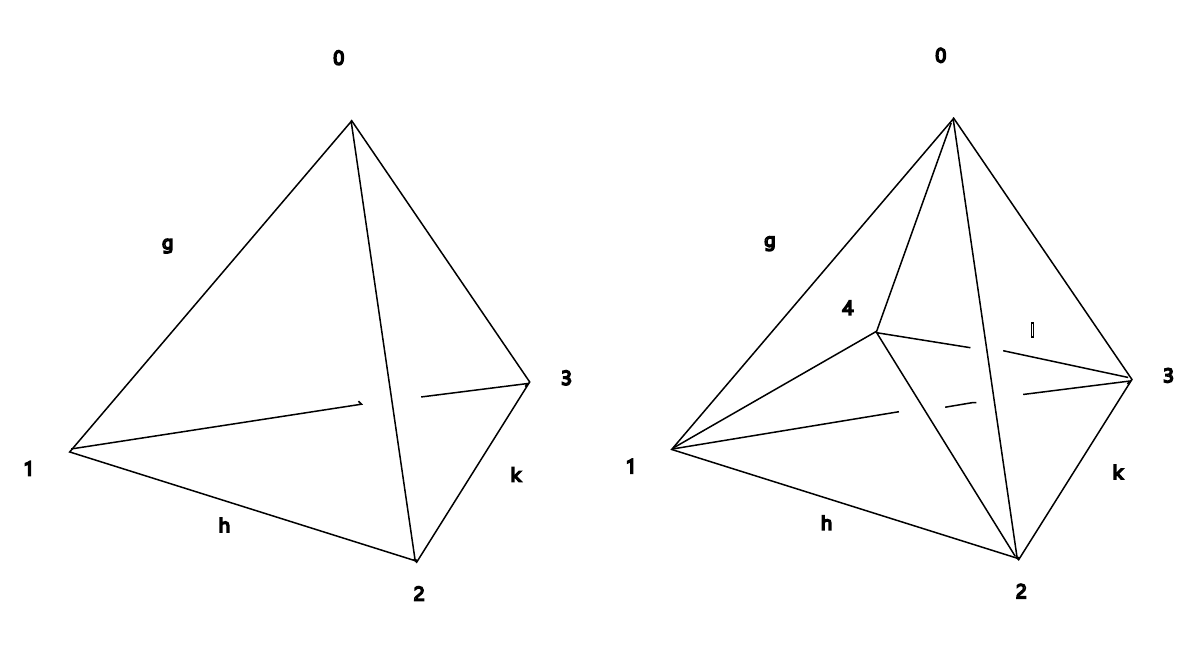}
\caption{Pachner 1-4 moves. The tetrahedron 0123 on the left becomes the tetrahedra 0124, 0234, 0134, and 1234 on the right.}\label{fig:Pachner2}
\end{figure}

For $n=3$, the Pachner moves are the Pachner 2-3 and 1-4 moves, see Fig.~\ref{fig:Pachner} and Fig.~\ref{fig:Pachner2}. The Pachner 2-3 move does not create a new vertex. The values of the 3-cocycle on the initial 2 3-simplices and the final 3 3-simplices satisfy the 3-cocycle condition. The 3-cocycle condition yields the invariance of the state sum \eqref{eq:DW} under the Pachner 2-3 moves. The Pachner 1-4 move creates a new vertex. Hence there is an additional factor $\frac{1}{|G|}$ in the state sum \eqref{eq:DW}. Also there is a new variable on an additional edge which runs through the elements of $G$. For each choice of the new variable, the values of the 3-cocycle on the initial 3-simplex and the final 4 3-simplices satisfy the 3-cocycle condition. Hence each summand in $\frac{1}{|G|}\sum$ equals the initial value of the 3-cocycle. Therefore the state sum \eqref{eq:DW} is invariant under the Pachner 1-4 moves.


The cocycle condition can be formulated in terms of volumes. One crucial observation is the following obvious lemma.

\begin{lemma}\label{volume}
Given an $(n+1)$-simplex $\Delta^{n+1}$ with vertices belonging to the space $\R^{n+1}$. Then the sum of volumes of 
the faces of codimension 1 of a projection of $\Delta^{n+1}$ onto the subspace $\R^n$ with appropriate signs is zero. 
\end{lemma}

Say, for $n=2$ one can consider a quadrilateral whose area can be split into two triangles in two ways (a quadrilateral is a projection of a tetrahedron), hence, the sum of two areas is equal to the sum of the other two areas.
We can articulate it as: the ``area'' function on $n$-simplices (in the latter case, $2$-simplices) gives zero on the ``boundary'' of an $(n+1)$-simplex ($3$-simplex). 

In Lemma \ref{volume}, we showed that a volume is just a cocycle. Though a cocycle need not be a volume, we may try to axiomatize volumes further and further to get to cocycles finally.
Therefore, $n$-cocycles can be treated as volumes of simplices in dimension $n$.

\begin{theorem}
The Dijkgraaf-Witten invariants are a partial case of the photography method.
\end{theorem}
\begin{proof}
We associate weights to $n$-simplices.
These weights satisfy the cocycle condition which means that 
the sum of weights of a boundary is zero.
We checked the cocycle condition.
This means that we checked that the sum of volumes is zero.

For example, we consider the case $n=3$.
Regarding the values of cocycles as volumes up to a sign corresponding to the orientation, the cocycle condition may be interpreted as the invariance of the sum of volumes, see Fig.~\ref{fig:Pachner} and Fig.~\ref{fig:Pachner2}. In the Pachner 2-3 move, 
the sum of the volumes of the tetrahedra 0123, 0124, and 0234 equals the sum of the volumes of the tetrahedra 0134 and 1234. This is exactly the cocycle condition (up to a sign)
$$\alpha[g|h|k]+\alpha[gh|k|l]+\alpha[g|h|kl]=\alpha[g|hk|l]+\alpha[h|k|l].$$
In the Pachner 1-4 move, the volume of the tetrahedron 0123 equals the sum of the tetrahedra 0124, 0234, 0134, and 1234. This is exactly the cocycle condition (up to a sign)
$$\alpha[g|h|k]=\alpha[g|h|kl]+\alpha[gh|k|l]+\alpha[g|hk|l]+\alpha[h|k|l].$$

Therefore, the data satisfy the data transmission law (the $n$-cocycle condition) for two adjacent triangulations connected by a Pachner move. The photography method constructs the Dijkgraaf-Witten invariants.

\end{proof}

Recall that in \cite[Sec. 3]{ManturovWanMay2023},
we took areas.
They are not really areas in some Euclidean space nor hyperbolic space.
They are formal variables, maybe from some field.
So, we can predict a solution by looking at some geometry in some proper sense. Say, area of the quadrilateral equals sum of areas of constituent parts, and then
we can check it by hand.
Here we do almost the same.
We take volumes, but volumes can not be negative, hence we take cocycles.
The only thing we need is that: 
a cocycle is linear and
its coboundary is 0.
This is what is guaranteed by volumes, but this may go beyond the volume in the strict sense.

\section{Further research}\label{sec:further}

The author is excited with the paper by Kevin Walker \cite{Walker}.
In that paper, he presents a ``universal'' state-sum invariant. In our opinion,
one of the greatest advantages of that paper is that it deals with cell decompositions/
handle slides and handle cancellations rather than with Pachner moves.

But there are a lot of similarities between that work and our work: 
we associate some {\em data} to cells, and this data is {\em naturally changed}
under {\em moves}. 

The next natural task is the
Crane-Yetter invariant, which we shall do in a subsequent publication.




It is important to calculate our invariants. Calculating Dijkgraaf-Witten invariants is hard and some special cases were calculated in \cite{de,Hu,Huang,Wan,Wang,WangWen,Wen}.

\end{document}